\documentclass{article}
\usepackage{graphicx} % Required for inserting images
\usepackage{algorithm}
\usepackage{algpseudocode}
\usepackage[a4paper, total={6in, 8in}]{geometry}

\usepackage{color}

\newcommand{\FCP}{\mathbf{FCP}}
\newcommand{\algo}{\mathbf{Algorithm}}
\newcommand{\Cut}{\mathbf{Cut}}
\newcommand{\et}{\eta_{\mathrm{conc}}}
\newcommand{\Mc}{M_{\FCP}}
\newcommand{\MF}{M_{\Cut}}
\newcommand{\Nc}{N_{\FCP}}
\newcommand{\NF}{N_{\Cut}}

\usepackage{amsmath}
\usepackage{amssymb}
\usepackage{amsfonts}
\usepackage{mathtools}
\usepackage{amsthm}
\newtheorem{theorem}{Theorem}[section]

\newtheorem{assumption}{Assumption}
\newtheorem{proposition}[theorem]{Proposition}
\newtheorem{corollary}[theorem]{Corollary}

\newtheorem{lemma}[theorem]{Lemma}
\newcommand{\alex}[1]{\textcolor{red}{[alex: #1]}}

\title{A simple and improved algorithm for noisy, convex, zeroth-order optimisation}
\author{Alexandra Carpentier\\
Institut f\"ur Mathematik, Universit\"at Potsdam}
\date{June 2024}

\begin{document}

\maketitle

\begin{abstract}
    In this paper, we study the problem of noisy, convex, zeroth order optimisation of a function $f$ over a bounded convex set $\bar{\mathcal X}\subset \mathbb{R}^d$. Given a budget $n$ of noisy queries to the function $f$ that can be allocated sequentially and adaptively, our aim is to construct an algorithm that returns a point $\hat x\in \bar{\mathcal X}$ such that $f(\hat x)$ is as small as possible.  We provide a conceptually simple method inspired by the textbook center of gravity method, but adapted to the noisy and zeroth order setting. We prove that this method is such that the $f(\hat x) - \min_{x\in \bar{\mathcal X}} f(x)$ is of smaller order than $d^2/\sqrt{n}$ up to poly-logarithmic terms. We slightly improve upon existing literature, where to the best of our knowledge the best known rate is in~\cite{lattimore2020improved} is of order $d^{2.5}/\sqrt{n}$, albeit for a more challenging problem. Our main contribution is however conceptual, as we believe that our algorithm and its analysis bring novel ideas and are significantly simpler than existing approaches. 
\end{abstract}

\section{Introduction}

We consider in this paper the setting of convex noisy zeroth-order optimisation. For $d\geq 1$, consider a bounded convex set $\bar{\mathcal X} \subset \mathbb R^d$ with non-zero volume, and consider a convex function $f:\bar{\mathcal X} \rightarrow [0,1]$. 

We consider a sequential setting with fixed horizon $n \in \mathbb N^*$. At each time $t \leq n$, the learner chooses a point $x_t\in \bar{\mathcal X}$ and observes a noisy observation $y_t\in [0,1]$ that is such that $\mathbb E[y_t|x_t = x] = f(x)$, and such that $y_t$ knowing $x_t$ is independent of the past observations.

In this work, we will study the problem of optimising the function $f$ in the sequential game described above, namely after the budget $n$ has been fully used by the learner, she has to predict a point $\hat x$ - based on all her observations $(x_t, y_t)_{t\leq n}$ - and her aim will be to estimate the minimum for the function $f$. Her performance for this task will be measured through the following (simple) \textit{regret}
$$f(\hat x) - \inf_{x\in \bar{\mathcal X}} f(x),$$
namely the difference between the true infimum of $f$, and $f$ evaluated at $\hat x$.

This setting, known as convex noisy zeroth-order optimisation, is related to two popular settings: first-order optimisation - where the learner has access to noisy evaluations of the (sub-)gradient of $f$ - and noiseless zeroth-order optimisation - where the noise $\epsilon_t$ is equal to $0$. We refer the reader to~\cite{nemirovskij1983problem, bubeck2015convex, hazan2016introduction, bertsekas2011incremental,lattimore2024bandit}, among others, for books and surveys on these topics. Unfortunately, a naive application of methods crafted for the two aforementioned topics to the problem of noisy zeroth-order optimisation typically provides poor results, as the noise present in the evaluations of the function perturbs significantly the learning process, and e.g.~makes attempts of computing (sub-)gradients of $f$ difficult - see e.g.~\cite{agarwal2011stochastic} for a precise discussion on this topic.

In the case where $d=1$, optimal algorithms however exist for this problem since a long time - see~\cite{nemirovskij1983problem} for a survey - and are related to dichotomic search. The optimal regret in this case is of order $n^{-1/2}$ up to polylogarithmic terms. An important question was then to extend this to the higher dimensional case; and while it is relatively simple to craft algorithms whose regret decays, up to logarithmic terms, as $\exp(cd)n^{-1/2}$ where $c>0$ is an universal constant, an important question that remained open for a long time was on whether the minimax regret was exponential with the dimension or not. A first ground-breaking work in this topic is to be found in~\cite[Chaper 9]{nemirovskij1983problem}, where they provide a complex algorithm whose regret can be bounded uniformly, with high probability, as $\frac{\mathrm{poly}(d)}{\sqrt{n}}$, proving that it is possible to have an algorithm whose regret depends actually only polynomially on $d$. This gave rise to a sequence of works, mostly in the related, more challenging setting where one aims at minimising the cumulative regret\footnote{In these works, the aim is to minimise the sum of collected samples - i.e.~sample as often as possible close to the minimum. They also often consider the challenging adversarial setting. Note that upper bounds in this setting morally yield upper bounds for our simper setting.} - see e.g.~\cite{agarwal2011stochastic,bubeck2017kernel, lattimore2020improved}. The exponent of the polynomial in $d$ has been successively reduced through this stream of works. The most actual algorithm and bound - to the best of our knowledge - is in~\cite{lattimore2020improved}, and for a more challenging problem (cumulative regret, adversarial setting). However their results would translate in our setting in a regret of order (up to logarithmic terms) $\frac{d^{2.5}}{\sqrt{n}}$. 
This has to be compared to the best lower bound, derived for this problem, which is of order $\frac{d}{\sqrt{n}}$, and which is proven over the smaller class of linear functions~\cite{shamir2013complexity}. This highlights the fact that a gap remains in this setting. In parallel, another stream of literature has been devoted to studying the effect of additional shape constraints, in particular strong convexity and smoothness - see e.g.~\cite{jamieson2012query, fokkema2024online} - under which the faster regret of order $\frac{d^{1.5}}{\sqrt{n}}$, is achievable. Note however that strong convexity is a very strong assumption that has important consequences - in particular, when combined to a smoothness assumption, it essentially implies that the shape of the level sets of $f$ is close to a ball. To complement this short litterature review, see rather~\cite{lattimore2024bandit} for an excellent very recent survey on these topics - see in particular~\cite[Section 2.3]{lattimore2024bandit} for a recent overview of the state of the art in these problems.
%So that this work is devoted to the setting of noisy zeroth-order convex optimisation without additional assumptions of strong convexity or smoothness.

In this paper, we provide a simple algorithm for the problem described above - solely under the additional assumption that the minimum of $f$ on $\bar{\mathcal{X}}$ is not too close to the border of $\bar{\mathcal{X}}$. We prove that with high probability and up to polylogarithmic terms depending on the probability, the budget, the dimension and the diameter of $\bar{\mathcal{X}}$, the regret is uniformly bounded as $\frac{d^2}{\sqrt{n}}$.
This slightly improves over the best known bound for this problem\footnote{Yet does not answer the open question on what is the minimax rate in this setting}. The main strength of this work, though, is the conceptual simplicity of the proposed algorithm, which contrasts with the complexity of existing approaches, and also its simple analysis. Indeed, our algorithm is an adaptation of the textbook center of gravity method~\cite{ayu1965algorithm, newman1965location}, namely a specific kind of dichotomic search, combined with an estimator of the gradient on a well-chosen proxy of $f$, at a well chosen point. In Section~\ref{sec:prel}, we present additional notations, as well as some preliminary results regarding these proxies of $f$, and also on estimating their values and gradients. In Section~\ref{sec:alg}, we provide the main algorithm, and the upper bound on its regret. All proofs are in the appendix, and are significantly commented for clarity.

%\begin{assumption}\label{ass:bound}
%Let $\epsilon>0$. We assume that at least one minima $x^*$ of $f$ on $\bar{\mathcal X}$ is such that its $l_2$ distance to the border of $\bar{\mathcal X}$ is larger than $\epsilon$, namely such that $x^* \in \bar{\mathcal X_\epsilon}$. 
%\begin{itemize}
%        \item at least one of its minima $x^*$ is such that it belongs to $\mathbb B_2(1)$. [Minimum not on the border of the ball $\mathbb B_2(1+\epsilon)$]
%    \item its restriction to $\mathbb B_2(1+\epsilon)$ takes values in $[0,1]$. [Boundedness within a ball]
%\end{itemize}
%We write $f(x^*) = f^*$.
%\end{assumption}

%For a given $\tau \geq x^*$, we want to find a point $\hat x \in \bar{\mathcal X}$ such that $\hat x \leq \tau$. We want to characterise the necessary budget $n$ of access to the function, in order to find such a  point with high probability. Without loss of generality, consider $\tau = 0$.

\section{Preliminary results and notations}\label{sec:prel}

 Write $(e_1,\ldots, e_d)$ for the canonical basis of $\mathbb R^d$. Write also for any Borelian set $\mathcal S \subset \mathbb R^d$, $\mathrm{vol}(\mathcal S)$ for the volume of this set (i.e.~its measure according to the Lebesgues measure), and $\mathrm{conv}(\mathcal S)$ for its convex hull. Let $p\geq 1$, for $R \geq 0$ and $x\in \mathbb R^d$, write $\mathbb B_p(x,R)$ for the $d-$th dimensional $l_p$ ball of radius $R$ and center $x$. We also write $\mathbb B_2(R) = \mathbb B_2(0,R)$, and $\mathbb S_2(R)$ for the $l_2$ sphere of center $0$ and radius $R$. For technical reasons, we will extend the definition of $f$ over $\mathbb R^d$, and write that for $x \not\in\bar{\mathcal X}$, $f(x) = +\infty$ - and we state by convention that when we sample a point $x_t \not\in\bar{\mathcal X}$, we obtain $y_t = +\infty$. We will say by convention that $f$ is convex on $\mathbb R^d$, as it is convex on $\bar{\mathcal X}$, and prolongated by $+\infty$ outside of $\bar{\mathcal X}$. We also state the following mild assumption on the function $f$, which implies that the minimum of $f$ cannot be too close to the border of $\bar{\mathcal X}$.

%, on the noise and on the function $f$.

%We consider a sequential setting with fixed horizon $n \in \mathbb N^*$. At each time $t \leq n$, the player selects a point $x_t\in \mathbb R^d$ and observes $y_t$.

%Consider a bounded convex set $\bar{\mathcal X} \subset \mathbb R^d$ with non-zero volume, and consider a convex function $f:\bar{\mathcal X} \rightarrow \mathbb R$. We write by convention that for $x \not\in\bar{\mathcal X}$, $f(x) = +\infty$R. % if $x_t \not\in \bar{\mathcal X}$
% Write $\bar{\mathcal X}^o$ for the border of $\bar{\mathcal X}$ and let $u>0$. Write also
%$$\bar{\mathcal X_u} = \{x \in \bar{\mathcal X}: \|x - \bar{\mathcal X}^o\|_2 \geq u\}.$$
%We consider a sequential setting with fixed horizon $n \in \mathbb N^*$. At each time $t \leq n$, the player selects a point $x_t\in \mathbb R^d$ and observes $y_t$.
%$$y_t = f(x_t) + \epsilon_t,$$
%where $\epsilon_t$ is a random noise.

%\begin{assumption}\label{ass:noise}
%    We assume that:
%    \begin{itemize}
%        \item  for any $t \leq n$, $y_t$, conditional to $(x_u)_{u \leq t}$ is independent of $(y_u)_{u \leq t}$.
%        \item $\mathbb E[ y_t|x_t=x] = f(x)$ for any $t\leq n, x\in \mathbb R^d$.
%        \item if $x_t \in \bar{\mathcal X}$, we have $y_t \in [-1,1]$ for any $t \leq n$.
%    \end{itemize}
%\end{assumption}

\begin{assumption}\label{ass:bound}
%    We assume that $f$ is bounded by $1$ on $\bar{\mathcal X}$. %$L$-Lipschitz on $\mathbb R^d$, and that $f$ is bounded by $M$ on $\bar{\mathcal X}$.

    Let $x^*$ be a minimum of $f$ on $\bar{\mathcal X}$ and write $f(x^*) = f^*$. Assume that $x^*$ is such that there exists $r >0$ such that $\mathbb B_2(x^*,r) \subset \bar{\mathcal X}$.

\end{assumption}

In what follows, we will consider some well-chosen proxies of $f$ which we will use in our algorithm. These proxies will be such that one can estimate in a "natural" way these proxies, as well as their gradients. We will study conditions under which these proxies have good properties. We follow here the natural idea - introduced in~\cite{bach2016highly} to the best of our knowledge for zeroth-order optimisation, and studied more generally in~\cite{akhavan2022gradient} - of considering a proxy of $f$ through smoothing in a neighborhood around each point. We will however adapt this neighborhood to some ambient convex set, as discussed below - and this adaptation is key for our algorithm later. In what follows, we first describe the proxies of $f$ that we will consider, and provide a condition under which the gradients of these proxies are informative regarding $f$ itself. We then explain how we can estimate these proxies and their gradient through noisy evaluations of $f$.

\subsection{Smoothed functional notations and results on smoothed convex functions}

Consider a convex subspace $\mathcal X \subset \bar{\mathcal X}$, of non-zero volume. We can define its barycenter
$$\mu_{\mathcal X} = \mathbb E_{X \sim \mathcal U_{\mathcal X}} X,$$
and its variance-covariance matrix
$$\Sigma_{ \mathcal X} = \mathbb V_{X \sim \mathcal U_{ \mathcal X}} X.$$
Since $\mathcal X$ has non-zero volume note that $\Sigma_{ \mathcal X}$ is invertible.

Write $F_{\mathcal X}$ for the linear transformation
$$F_{\mathcal X}: x \rightarrow \frac{1}{\sqrt{d}}\Sigma_{ \mathcal X}^{-1/2}(x - \mu_{ \mathcal X}).$$
Note that the convex set $\mathcal Z^{\mathcal X} = F_{\mathcal X} (\mathcal X)$ is in isotropic position renormalised by $d^{-1/2}$. Write also $ \bar{\mathcal Z}^{\mathcal X} = F_{\mathcal X}(\bar{\mathcal X})$, and 
$z^* = F_{\mathcal X} (x^*)$.

Define for any $z\in \mathbb R^d$
$$g^{\mathcal X} (z) = f(F^{-1}_{\mathcal X}(z)) =f(\sqrt{d}\Sigma_{ \mathcal X}^{1/2}(z+\mu_{ \mathcal X})).$$
Note that $g^{\mathcal X}$ is convex on $\mathbb R^d$ and that also in particular the function $f$  is the same up to a linear transformation than the function $g^{\mathcal X}$ - and this linear transformation transforms $\bar{\mathcal X}$ in $\bar{\mathcal Z}^{\mathcal X}$ and $x^*$ into $(z^*)^{\mathcal X}$. When no ambiguity arises, we write $g$ for $g^{\mathcal X}$, $z^*$ for $(z^*)^{\mathcal X}$, $\mathcal Z$ for $\mathcal Z^{\mathcal X}$ and $\bar{\mathcal Z}$ for $\bar{\mathcal Z}^{\mathcal X}$  - and note that $g(y^*) = f^*$.

%Define $\bar{\mathcal C}(c)$ for the set of all $x\in \bar{\mathcal C}$ such that $\mathbb B_2(x,c) \subset \bar{\mathcal C}$. Note that $\bar{\mathcal C}(c)$ is convex.

Define for $c> 0, z\in \mathbb R^d$
$$g_c^{\mathcal X} (z) = \mathbb E_{Z\sim \mathcal U_{\mathbb B_2(c)}}g(z+Z),$$
with the convention $g_0^{\mathcal X} (.) = g^{\mathcal X}$. Again when no ambiguity arises, we write $g_c$ for $g_c^{\mathcal X}$. Note that $g_c$ is convex on $\mathbb R^d$, and that $g_c \geq g_{c'}$ for any $0 \leq c' \leq c$. Note also that for $c>0$, $g_c$ is differentiable on $\mathbb R^d$, and that by Stoke's theorem, for any $z\in \mathbb R^d$:% such that $\mathbb B_2(z,c) \subset \bar{\mathcal Z}$:
\begin{equation}\label{eq:nableq}
    \nabla g_c(y) = \frac{d}{c^2}  \mathbb E_{Z\sim \mathcal U_{\mathbb S_2(c)}} \Big[Zg(y+Z)\Big],
\end{equation}
see~\cite[Theorem 5]{akhavan2022gradient} for a precise reference.%, and see~\cite{nemirovskij1983problem} for the first paper introducing this gradient estimator to the best of our k

A fundamental property of convex functions is that, for any $z\in \mathbb R^d$ and any sub-gradient $\nabla g(z)$ at this point, if $g(z)) - g(\tilde z)$ is large, the sub-gradient correlates significantly with $z-\tilde z$. Namely
$$\langle \nabla g(z), z - \tilde z \rangle \geq g(z)) - g(\tilde z).$$
The following lemma is a simple, yet key result for this paper, and extends this property to the smoothed function $g_c$ - namely, that if $g(z)) - g(\tilde z)$ is large, the sub-gradient $\nabla g_c(z)$ correlates significantly with $z-\tilde z$ - in fact it holds under the relaxed condition that $g_c(z)) - g(\tilde z)$ is large.
%It states a condition under which, for a point $g\in \mathbb R^d$, the gradient of the smoothed function at this point $\nabla g_c(z)$ is correlated 
\begin{lemma}\label{lem:geom}
    %Assume that $x^* \in \mathcal X$ and 
    Let $c>0$ and $z,\tilde z\in \mathbb R^d$. %\bar{\mathcal Z}$ such that $\mathbb B_2(z,2c) \subset \bar{\mathcal Z}$.
    If $g_{2c} (z) - g_c(z) \leq 2^{-2}[g_c(z)) - g(\tilde z)]$%, then for any $x'$ such that $g(x') - f^* \leq  2^{-3}[g_c(x)) - f^*]$ % for any $\bar x \{u: f(u) \leq \epsilon g(x)\}$, we have:
    $$\langle \nabla g_c(z), z - \tilde z \rangle \geq \frac{3}{4}[g_c(z)) - g(\tilde z)].$$
\end{lemma}
The proof of this lemma is in Appendix~\ref{lem:tech}. It implies in particular that if $g_{2c} (z) - g(z) \leq 2^{-2}[g(z)) - f^*]$ - i.e.~if the distance between the proxy $g_{2c}(z)$ and the function $g(z)$ is of smaller order than the optimality gap of $g(z)$ (compared to the minimum $f^*$ of $g$), then the gradient of the proxy is interesting, namely $\nabla g_c(z)$ is correlated to $z - z^*$, with correlation larger than said optimality gap. In other words, the properties of $\nabla g_c(z)$ are similar to those of a sub-gradient $\nabla g(z)$, when it comes to the minimal correlation to the direction of the minimum.

\subsection{Estimators of the function and of the gradient of smoothed convex functions}

Consider now $z \in \bar{\mathcal X}$ and resp.~$Z_1^{(b)},\ldots, Z_N^{(b)} \sim_{i.i.d.} \mathcal U_{\mathbb B_2(c)}$ and $Z_1^{(s)},\ldots, Z_N^{(s)} \sim_{i.i.d.} \mathcal U_{\mathbb S_2(c)}$ for points sampled respectively uniformly in the ball of center $0$ and radius $c$, and in the sphere of center $0$ and radius $c$. Assume that we observe independent noisy observations of the function $f$ at the points $F^{-1}_{\mathcal X}(z+Z_1^{(k)}),\ldots, F^{-1}_{\mathcal X}(z+Z_N^{(k)})$ where $k \in \{b,s\}$ - i.e.~equivalently we observe independent noisy observations of the function $g$ at the points $z+Z_1^{(k)},\ldots, z+Z_N^{(k)}$ - that we write
$$(\tilde y_t^{(k)})_{t\leq N},$$
where the $\tilde y_t^{(k)} \in [0,1]$ are such that $\mathbb E[\tilde y_t^{(k)}|F^{-1}_{\mathcal X}(z+Z_t^{(k)}) = x] = f(x)$ and such that $\tilde y_t^{(k)}$ knowing $F^{-1}_{\mathcal X}(z+Z_t^{(k)})$ is independent of the past observations.

Define:
\begin{equation}\label{eq:estimg}
    \hat g_c(z) = \frac{1}{N} \sum_{i=1}^N\tilde y_i^{(b)},
\end{equation}
and
\begin{equation}\label{eq:nablag}
\widehat{ \nabla g_c}(z) = \frac{d}{c^2N} \sum_{i=1}^N Z_i^{(s)}\tilde y_i^{(s)}.
\end{equation}
%Write also
%$$(v_c^{\mathcal X})^2 (x) = \mathbb V_{Z\sim \bar{\mathcal C}(c)}g(x+Z),$$
%which we will write $v_c^2(x)$ if there is no ambiguity. Write also $m_c(x) = m_c^{\mathcal X}(x) = \max_{u \in \mathbb B_2(c)} g(x+u)$ is the maximal (half)-range of the random variable $g(x+Z_i)$ for $Z_i$ defined as above. Note that by Assumption~\ref{ass:bound}, we have $v_c^2(x) \leq 4m_c(x)^2$. Note also that $m_c(x)\leq Lc \mathrm{diam}(\mathcal X)$ $\mathrm{diam}(\mathcal X)$ is the diameter of $\mathcal X$.

%as $f$ takes values in $[0,1]$ on $\bar{\mathcal C}$.

The following lemma provides a concentration result for both the estimator of the function, and of the estimator of the gradient.
\begin{lemma}\label{lem:conc}
    %Assume that Assumption~\ref{ass:noise} holds. 
    Let $c\geq 0$, $z\in \bar{\mathcal Z}$ such that $\mathbb B_2(z, c) \subset \bar{\mathcal Z}$ and $u \in \mathbb R^d$. With probability larger than $1-\delta$
%    $$|\hat g_c(x) - g_c(x)| \lesssim \frac{(\sigma \lor v_c(x))}{\sqrt{N}} \sqrt{\log(1/\delta)} + m_c(x)\frac{\log(1/\delta)}{N} ,$$
    $$|\hat g_c(z) - g_c(z)|  \leq \frac{\et(1/\delta)}{\sqrt{N}},$$
    and if $N \geq d\log(2/\delta)$
%    $$|\langle \widehat{ \nabla g_c}(x) - \nabla g_c(x), u\rangle| \lesssim \|u\|_2  \frac{\sqrt{d}}{c} \Big[ \frac{ (\sigma \lor v_c(x))}{\sqrt{N}} \sqrt{\log(1/\delta)} + m_c(x)\frac{\log(1/\delta)}{N} \Big].$$
$$|\langle \widehat{ \nabla g_c}(z) - \nabla g_c(z), u\rangle|  \leq \et(1/\delta) \|u\|_2  \frac{\sqrt{d}}{c\sqrt{N}},$$
where $\et(1/\delta) = 4 \sqrt{ \log(2/\delta)}$.
%    where $m$ is the maximum of $g$ on $\{z: z=z_1+z_2, z_1 \in \mathcal C, z_2 \in \mathbb B_2(c)\}$.
\end{lemma}
The proof of this lemma is in Appendix~\ref{lem:tech}, and is based on very standard concentration arguments. The study of related estimators was first formulated to the best of our knowledge in~\cite{nemirovskij1983problem}, and then refined in~\cite{bach2016highly, akhavan2022gradient} (among others). Note however that in these works, the proximity of these estimators to $g$ or its gradient is controlled, under smoothness assumptions. This is not the approach that we take here, as we do not work under additional smoothness assumptions - so that the proxies $g_c$ can be arbitrarily far from $g$ and its gradient in many points.

%Write $\bar \sigma = \sigma \lor (Lc \mathrm{diam}(\mathcal X))$.
%\begin{corollary}
%    Let $x\in \mathbb R^d$ and $u \in \mathbb R^d$. With probability larger than $1-\delta$
%    $$|\hat g_c(x) - g_c(x)| \lesssim \frac{\bar \sigma}{\sqrt{N}} \sqrt{\log(1/\delta)},$$
%    and
%    $$|\langle \widehat{ \nabla g_c}(x) - \nabla g_c(x), u\rangle| \lesssim \|u\|_2  \frac{\sqrt{d}}{c}  \frac{ \bar \sigma}{\sqrt{N}} \sqrt{\log(1/\delta)} .$$
%\end{corollary}

\section{Algorithm}\label{sec:alg}

Our algorithm is an adaptation of the center of gravity method to the noisy, gradientless case. In the classical  center of gravity method, we iteratively refine the convex set $\bar{\mathcal{X}}$ at each step. More precisely, assume that we are given a convex set $\mathcal X \subset \bar{\mathcal{X}}$ at a given iteration. We refine it by computing the gradient $\nabla(f)(x)$ of $f$ at the center of gravity $x$ of $\mathcal X$, and updating $\mathcal X$ to $\mathcal X \cap \{u: \langle \nabla f(x), u-x\rangle \leq 0\}$. This method is efficient as
\begin{enumerate}
    \item by convexity of $f$, $x^*$ remains in $\mathcal X$ for any iteration, and
    \item a fundamental property of convex sets is that if we separate them in two parts by any hyperplane going through their center of gravity, both part of the convex set have approximately the same volume.
\end{enumerate}

In our case, we do not have access to $\nabla(f)$, but only to noisy evaluations of $f$. The idea behind our method is to estimate instead the gradient of another function - namely, of $g_c^{\mathcal X}$ for a well chosen $c$, i.e.~a linear transformation of $f$ that is also smoothed. We have seen in Lemma~\ref{lem:conc} that this task can be performed efficiently. However, this gradient might be quite different from any sub-gradient of $f$. We have however seen in Lemma~\ref{lem:geom} that under the condition that $g_{2c}(0)-g_{c}(0)$ is small enough, the gradient of $g_c$ has the nice property that it correlates positively to $F_{\mathcal X}(x)-F_{\mathcal X}(\tilde x)$, for any $\tilde x$ such that $f(\tilde x)$ is small enough. So that $F_{\mathcal X}^{-1}(\nabla g_c(x))$ could be used instead of the gradient of $f$ in the center of gravity method. 

The only problem remaining that the center of gravity is not necessarily such that $g_{2c}(0)-g_{c}(0)$ is small. In order to circumvent this, we find another point $z$ that is such that it has this property, and is also such that $\|z\|_2$ is small enough so that cutting $\mathcal X$ in $F_{\mathcal X}^{-1}(z)$ enjoys provides similar volume guarantees than cutting it in $x$.

The main algorithm $\algo$ described below in Figure~\ref{Algo:algo} is therefore using two recursive sub-routines:
\begin{itemize}
    \item it first calls an iterative sub-routine $\Cut$ described in Figure~\ref{Algo:Cut} that \textit{cuts the current set $\mathcal X$ in two}, until the budget is elapsed,
    \item this routine calls another sub-routine $\FCP$ described in Figure~\ref{Algo:FCP}, which finds a \textit{good cutting point}, as explained above.
\end{itemize}

\subsection{Part 1: finding a cutting point}

We first describe the sub-routine that identifies a good candidate for a cutting point. This subroutine acts in the linear transformation $\mathcal Z^{\mathcal X}$ of $\mathcal X$ through $F_{\mathcal X}$.  Starting from $z_0$, we want to find using a budget of order $N$ - up to multiplicative polylog terms - a point $z$ such that:
\begin{itemize}
    \item either $g_{2c}(z) - g_{z}( x) \leq 2^{-3} (g_{c}(z) - f^*)$, or $g(z)$ is small (say, smaller than $1/\sqrt{N}$ up to multiplicative polylog terms)
    \item $\| z -  z_0\|_2$ is of smaller order than $c$ up to multiplicative polylog terms,
\end{itemize}
%using as little budget as possible
provided that such a point exists. In this way, we ensure that this point would satisfy the condition of Lemma~\ref{lem:geom}, or be such that $g(z)$ is small enough, and also that it is not too far from $z_0$.

%We will prove that such a point exists and that we can find it using less than app$d/g_{c}(\tilde x_0)^2$ samples. 

Assume that we are given a set $\mathcal X$ and $c >0$. For $N\geq 1$, let $I_N = \log_2(N)+1$ and $\Mc(N) = \log(2N)/\log(17/16) +1 $. %\alex{TODO}% and $U = ...$.
The recursive algorithm $\FCP$ that performs this takes as parameters a candidate for a cutting point $z \in  \mathbb R^d$, the current set to be cut $\mathcal X \subset \mathbb R^d$, a smoothness parameter $c>0$, a basis number of samples that will be our approximate final budget up to polylog terms $N \in \mathbb N$, a counting of the number of recursive rounds performed $s \geq 0$, and a confidence parameter $\delta>0$. During each run, the algorithm either returns the final cutting point $z\in \mathbb R^d$, as well as an estimator of $g(z)$ by $\hat g_z)$, or calls itself recursively. Note that this subroutine will require to sample the function $f$ and as it is typically called by another algorithm which operates based on a total budget $n$, as soon as this budget is elapsed, the algorithm $\FCP$ terminates returning the current $(z,\hat g_z)$. It proceeds in the following steps.
%If $\mathcal X$ is not a subset of $\bar{\mathcal X}$ or is not convex, the algorithm terminates, and if $\mathbb B_2(z,2c)$ is not a subset of $\bar{\mathcal X}$, teh algorithm terminates.
\begin{enumerate}
    \item It first sample the function $f$ in $F_{\mathcal X}^{-1}(z)$ for $N$ times and estimate in this way $g^{\mathcal X}(z)$ by $\hat g_z$ as described in Equation~\eqref{eq:estimg}.%until the first number of samples $N_x$ such that $\hat g(x) \gtrsim \bar \sigma \sqrt{\log(n/\delta)/N_x}$ - estimation as described above with $N_x$ samples
 %   \item Write $U = N_x$
    \item For all integer $i \leq I_N$, sample $2^i$ points distributed as $z+\mathcal U_{\mathbb B_2(2c)}$, and write $(z_j^{(i)})_{i \leq I_N, j \leq 2^i}$ for these points. Sample $2^{-i}N/i^2$ times the function $f$ at $F_{\mathcal X}^{-1}(z_j^{(i)})$ and estimate in this way $g^{\mathcal X}(z_j^{(i)})$ by $\hat g_{z_j^{(i)}}$ as described in Equation~\eqref{eq:estimg}.
    \item If there exists $z_j^i$ such that $\hat g_{z_j^{(i)}}  - \hat g_z \geq  \frac{(17/16)^s}{16N}+4\et(2^i i^2 \Mc(N)/\delta) \sqrt{\frac{i^2 2^i}{N}}$ then call $\FCP(z_j^i, \mathcal X, c,N, s+1, \delta)$. Otherwise return $(z,\hat g_z)$.
\end{enumerate}
In this way, we evaluate whether, in a radius of $2c$ around $z$ there is a significantly large set of points such that $g$ evaluated in these points is large - i.e.~exponentially growing with the number of iterations $s$. If this is the case, we identify one of these points, and propose it as next barycentric candidate. Otherwise, we identify $z$ as a good candidate and return it. The full algorithm is summarized in Figure~\ref{Algo:FCP}

\begin{algorithm}[ht]
	\caption{$\FCP$}\label{Algo:FCP}
	\begin{algorithmic}[1]
		\Require $(z, \mathcal X, c, N, s,\delta)$
		\Ensure $(z,\hat g_z)$ - except if the budget elapses in which case it stops
		\State Sample the function $f$ in $F_{\mathcal X}^{-1}(z)$ for $N$ times and estimate in this way $g^{\mathcal X}(z)$ by $\hat g_z$ as in Equation~\eqref{eq:estimg}
  \For{all integer $i \leq I_N$} 
  \State Sample $2^i$ points as $z+\mathcal U_{\mathbb B_2(2c)}$, and write $(z_j^{(i)})_{i \leq I, j \leq 2^i}$ for these points
  \State Sample $2^{-i}N/i^2$ times the function $f$ at $F_{\mathcal X}^{-1}(z_j^{(i)})$ and estimate in this way $g^{\mathcal X}(z_j^{(i)})$ by $\hat g_{z_j^{(i)}}$ as in Equation~\eqref{eq:estimg}
  \EndFor
     \If{there exists $z_j^i$ such that $\hat g_{z_j^{(i)}}  - \hat g_z \geq  \frac{(17/16)^s}{16N}+4\et(2^i i^2 \Mc(N)/\delta) \sqrt{\frac{i^2 2^i}{N}}$} 
    \State  call $\FCP(z_j^i, \mathcal X, c,N, s+1, \delta)$ for such a $(i,j)$
    \Else 
    \State {\bf return} $(z,\hat g_z)$
    \EndIf
	\end{algorithmic}
\end{algorithm}

\subsection{Part 2: routine for effectively cutting the space}

We now describe the subroutine that iteratively cuts the space, taking as parameter a convex set $\mathcal X \subset \bar{\mathcal X}$. It also maintains a current estimation $\hat x$ of the minimum. It updates these to $\mathcal X', \hat x'$. We would like it to satisfy that with high probability:
\begin{itemize}
    \item the volume of $\mathcal X'$ is a fraction of the volume of $\mathcal X$
    \item {\bf either} a small ball around the true minimum is in $\mathcal X'$, {\bf or} the current estimator of the minimum $\hat x'$ is already very good.
\end{itemize}

%Set $c_i = 2^{-i}d^{-1/2}$ and $I = 2 \lfloor \log n\rfloor +1$. Set $\gamma = (\log n)^2$, and $N = $.
Set $\NF = \frac{n\log(10/9) }{4d \log\Big(nd\mathrm{diam}(\bar{\mathcal X})/r\Big)}$, $\MF = 5n/\NF$, $\Nc =  \NF/(4\Mc)$ where\footnote{Looking at the definition of $\Mc(\Nc)$, it is clear that such $\Nc,\Mc$ exist and $\Nc$ is of order $\NF$ and $\Nc$ is of order $\log(\NF)$).} $\Mc = \Mc(\Nc)$, and $c =  1/(8e\Mc\sqrt{d})$.
%, , and , and $\Nc =  \NF/(4\Mc)$ where $\Mc = \Mc(\Nc)$. %\alex{TODO}
We define the recursive algorithm $\Cut$ taking as parameters a candidate set $\mathcal X \subset \mathbb R^d$, a candidate estimator of the minimum of $f$ by $\hat x \in \mathbb R^d$, an estimate of the value of $f$ at this point $\hat f \in \mathbb R$,  and a probability $\delta>0$. During each run, the algorithm calls itself recursively. Note that this subroutine will require to sample the function $f$ and as it is typically called by another algorithm which operates based on a total budget $n$, as soon as this budget is elapsed, the algorithm $\Cut$ terminates returning the current $\hat x$. It proceeds in the following steps.
\begin{enumerate}
%    \item Compute $\mu_T = \mu_{\mathcal X_T}$, $\Sigma_T = \Sigma_{\mathcal X_T}$
    \item Run $\FCP(0,\mathcal X, c,\Nc,0,\delta)$ and collect $(z,\hat g_z)$
    \item If $\hat g_z \leq \hat f$ set $\hat x' = F^{-1}_{\mathcal X}(z)$ and $\hat f' = \hat g_z$, otherwise set $\hat x' = \hat x$ and $\hat f' = \hat f$
    \item Compute an estimator $\widehat{ \nabla g_{c}}$ of $\widehat{ \nabla g^{\mathcal X}_{c}}(x)$ using $\NF$ samples, as described in Equation~\eqref{eq:nablag}.
    \item Set $\mathcal X' = \mathcal X \cap F_{\mathcal X}^{-1}\Big(\{u: \langle u-z, \widehat{ \nabla g_{c}} \rangle \leq 0 \}\Big)$
    \item Run $\Cut(\mathcal X',\hat f', \hat x',\delta)$
\end{enumerate}
This follows the idea of the center of gravity method, using a well-chosen cutting point returned by $\FCP$  and cutting then according to the gradient of a smoothed version of $f$, and continuing recursively. The full algorithm is summarised in Figure~\ref{Algo:Cut}.
 
%Write $\bar{\mathcal X}^+ = \{y: \exists x\in \bar{\mathcal X}~~~\mathrm{s.t.}~~~\|y - x\|_2 \leq 1/(Ln)$.

\begin{algorithm}[ht]
	\caption{$\Cut$}\label{Algo:Cut}
	\begin{algorithmic}[1]
		\Require $(\mathcal X, \hat f, \hat x, \delta)$
		\Ensure $\hat x$ as the budget elapses - after all $n$ samples have been used, it returns the current $\hat x$
		\State Run $\FCP(0,\mathcal X, c,\Nc,0,\delta)$ and collect $(z,\hat g_z)$
        \If{$\hat g_z \leq \hat f$} 
         \State Set $\hat x' = F^{-1}_{\mathcal X}(z)$ and $\hat f' = \hat g_z$
         \Else 
         \State Set $\hat x' = \hat x$ and $\hat f' = \hat f$
         \EndIf
         \State Compute an estimator $\widehat{ \nabla g_{c}}$ of $\widehat{ \nabla g^{\mathcal X}_{c}}(x)$ using $\NF$ samples, as in Equation~\eqref{eq:nablag}
    \State Set $\mathcal X' = \mathcal X \cap F_{\mathcal X}^{-1}\Big(\{u: \langle u-z, \widehat{ \nabla g_{c}} \rangle \leq 0 \}\Big)$
    \State Run $\Cut(\mathcal X',\hat f', \hat x',\delta)$
	\end{algorithmic}
\end{algorithm}

\subsection{Part 3: final algorithm}

The main algorithm $\algo$ is finally launched with a total budget $n$ and a confidence parameter $\delta>0$, and returns an estimator $\hat x$ of the minimum. It is basically an application of $\Cut$ on a reasonable initialisation, and proceeds in the following steps.
\begin{enumerate}
    \item we first sample $\Nc$ times the function $f$ at $\mu_{\bar{\mathcal X}})$ and compute an estimator $\hat f$ of $f(\mu_{\bar{\mathcal X}}))$ as in Equation~\eqref{eq:estimg} - recalling that $f(\mu_{\bar{\mathcal X}})) = g(0)$.
    \item we apply $\Cut(\bar{\mathcal X}, \hat f, \mu_{\bar{\mathcal X}}, \delta)$) and retrieve $\hat x$ when the budget is elapsed.
    \item we return $\hat x$
\end{enumerate}
This algorithm is summarised in Figure~\ref{Algo:algo}.

\begin{algorithm}[ht]
	\caption{$\algo$}\label{Algo:algo}
	\begin{algorithmic}[1]
		\Require $(n,\delta)$
		\Ensure $\hat x$ as the budget elapses
		\item Sample $\Nc$ times the function $f$ at $\mu_{\bar{\mathcal X}}$ and compute in this way an estimator $\hat f$ of $f(\mu_{\bar{\mathcal X}}))$ as in Equation~\eqref{eq:estimg}
    \State Run $\Cut(\bar{\mathcal X}, \hat f, \mu_{\bar{\mathcal X}}, \delta)$ until the budget is elapsed and retrieve $\hat x$
    \State {\bf return} $\hat x$ %(computed during the iterative runs of $\Cut$) and {\bf return} it
	\end{algorithmic}
\end{algorithm}

%We apply $\Cut(\bar{\mathcal X}^+, +\infty, \mu_{\bar{\mathcal X}})$) and let it run into a recursive loop until budget is finished, and we return the current $\hat x$ as our estimator of the minimum. %At the end of the budget we return the projection $\hat x_P$ of the current value of $\hat x$ on $\bar{\mathcal X}$.

The following theorem holds for the output of $\algo$.
\begin{theorem}\label{thm:cut}
    Assume that Assumptions~\ref{ass:bound} holds. The algorithm $\algo$ launched with a total budget $n$ and a confidence parameter $\delta$ returns $\hat x$ that is such that with probability larger than $1-\delta$:
    \begin{align*}
    f(\hat x) - f^* &\leq \Big[2^{16} \et(\Mc/\delta)\log(2\Mc/\delta) \frac{1}{\sqrt{\Nc}}\Big]\\ &\lor \Big[32\et(10d\MF/\delta)  \frac{d}{c\sqrt{\NF}}\Big] \lor (8/n) \lor [d \log(2/\delta)/\NF]\\
    &\leq  c\mathrm{poly}\log\big(nd\mathrm{diam}(\bar{\mathcal X})/r\big)^\alpha \times \frac{d^2}{\sqrt{n}} \log(1/\delta)^{3/2},
    \end{align*}
    where $c,\alpha>0$ are two absolute constants (independent on $f,\bar{\mathcal X}, n,d,\delta$).%, and where $\mathrm{poly}\Big[\log\big(nd\mathrm{diam}(\bar{\mathcal X})/r\big)$ is a polytermthat is polylogarithmic in  
\end{theorem}
This theorem is proved in Subsection~\ref{proof:thmain} and its proof is commented and explained therein. The regret depends only logarithmically on $\mathrm{diam}(\bar{\mathcal X})$ and on the diameter  $r$ of a ball centered around the minimum and contained in $\bar{\mathcal X}$ - which is not surprising and already observed in past works. Up to logarithmic terms, our regret here is of order $d^2/\sqrt{n}$ which slightly improves with respect to an adaptation of the best known bound in~\cite{lattimore2020improved} - which is derived for the more challenging problem of adversarial minimisation of the cumulative regret\footnote{We believe that our algorithm can be easily modified to accommodate cumulative regret in the stochastic case, and have a cumulative regret of order $d^2\sqrt{n}$. We however do not think that it could be easily adapted to the adversarial case.}, but which could translate in our setting as being of order $d^{2.5}/\sqrt{n}$. Beyond this slight improvement, the main strength of our approach is in terms of our algorithm and proof technique, which are - we believe - significantly simpler than existing results\footnote{However, while our algorithm is simple conceptually, it is extensive computationally as it requires an (approximate) computation of barycenters of successive convex sets, which is typically very costly.}. We hope that these techniques could maybe be refined to develop a tighter understanding of this problem, and evolve toward understanding the minimax regret in this problem.

\paragraph{Acknowledgements.} We would like to thank very warmly Evgenii Chzhen, Christophe Giraud, and Nicolas Verzelen for many insightful discussions on this problem, for their valuable opinion, and for their support without which this work would not have been written.\\
This work is partially supported by the Deutsche Forschungsgemeinschaft (DFG) CRC 1294 'Data Assimilation', Project A03, by the DFG Forschungsgruppe FOR 5381 "Mathematical Statistics in the Information Age - Statistical Efficiency and Computational Tractability", Project TP 02, by the Agence Nationale de la Recherche (ANR) and the DFG on the French-German PRCI ANR ASCAI CA 1488/4-1 "Aktive und Batch-Segmentierung, Clustering und Seriation: Grundlagen der KI".

\bibliographystyle{plain} % We choose the "plain" reference style
\bibliography{lib}

\section{Proofs of the results in this paper}

\subsection{Proof of Theorem~\ref{thm:cut}}\label{proof:thmain}

Assume first that $\NF \leq  d \log(2/\delta)$. Then by definition of $\NF$ it means that $1 \leq d \log(2/\delta)/\NF$, so that the bound in Theorem~\ref{thm:cut} is trivially satisfied for any $\hat x \in \bar{\mathcal X}$. From now on, we therefore restrict to the converse case where $\NF \geq  d\log(2/\delta)/\NF$ - so that the second part of Lemma~\ref{lem:conc} can be applied to gradients constructed with $\NF$ points, as we do in our algorithm.

%\alex{This proof is commented quite a lot}

\paragraph{Step 1: Definition of a near-optimal set and lower bound on its volume.} Write
$$\mathcal X^* = \{x^*, x^*+re_i/n, x^*-re_i/n, i\leq d\}.$$
We first state a lemma ensuring that under Assumption Assumption~\ref{ass:bound}, the convex hull of $\mathcal X^*$ is in $\bar{\mathcal X}$, and that the volume ration between this convex hull and the volume of $\bar{\mathcal X}$ is lower bounded.
\begin{lemma}\label{lem:vol}
Assume that Assumption~\ref{ass:bound} holds. It holds that
$$\mathcal X^* \subset \mathrm{conv}(\mathcal X^*)\subset \bar{\mathcal X},$$
and
$$\frac{\mathrm{vol}(\mathrm{conv}(\mathcal X^*))}{\mathrm{vol}(\bar{\mathcal X})} \geq \Big[\frac{r}{nd\mathrm{diam}(\bar{\mathcal X}}\Big]^d.$$
\end{lemma}
Note that by convexity of $f$ and by definition of $\mathcal X^*$, we have that for any $u\in \mathrm{conv}(\mathcal X^*)$, $f(u) - f^* \leq 1/n$. The above lemma lower bounds the volume ratio $\frac{\mathrm{vol}(\mathrm{conv}(\mathcal X^*))}{\mathrm{vol}(\bar{\mathcal X})}$.

\paragraph{Step 2: Results on $\FCP$.} The following result holds for algorithm $\FCP$.
\begin{proposition}\label{prop:cut}
    Assume that $\mathbb B_2(z_0, 2\Mc c) \subset \bar{\mathcal Z}^{\mathcal X} $. With probability larger than $1 - 4\delta$: $\FCP(z_0,\mathcal X, c,N,0,\delta)$ returns $z$ such that
    \begin{itemize}
    \item {\bf either} 
    $$g_{2c}(z) - g_c( z) \leq 2^{-3} (g_c( z) - f^*),$$ {\bf or} 
    $$g( z) - f^* \leq 2^{15} \et(\Mc/\delta)\log(2\Mc/\delta) \frac{1}{\sqrt{N}},$$
    %$g_{2c}(z) - g_c( z) \leq g_{2c}( z) - g( z) \leq 2^{-3} (g(z)  - f^*)\leq 2^{-3} (g_c( z) - f^*)$
    \item $|g( z)  - \hat g_z| \leq \et(\Mc/\delta)/\sqrt{N}$,
    \item $\| z -  z_0\|_2 \leq 2\Mc c$,
    \item the total budget $T_{\FCP}$ used to find $z$ is smaller than $4 \Mc N$, so that
    $N \leq T_{\FCP} \leq 4 \Mc N$.
\end{itemize}
\end{proposition}
The main idea behind this result is that on a high probability event:
\begin{itemize}
    \item if a point $z_j^{(i)}$ is selected for being a candidate for a cutting point, then it means that $g(z_j^{(i)})$ is larger than a quantity growing exponentially with the number of iterations $s$. As the range of $g$ is bounded on $\mathcal Z^{\mathcal X}$, this means that the number of recursive calls to $\FCP$ should be logarithmically bounded - hence the bound on $\|z - z_0\|_2$ and the bound on the number of samples used.
    \item if none of the $z_j^{(i)}$ is selected for being a candidate for a cutting point, then it either means that (i) they are all small, and as they are representative of the average value of $g$ on $\mathbb B_2(z,2c)$, then $g_{2c}(z)$ will be small enough to satisfy our condition in Lemma~\ref{lem:geom}, or (ii) that $g(z)$ is already very small.
\end{itemize}

\paragraph{Step 3: Results on a single run of $\Cut$.} We now state the following lemma that describes the high probability behaviour of $\Cut$, provided that it is given a reasonable set of parameters. Set 
$$B = \Big[2^{16} \et(\Mc/\delta)\log(2\Mc/\delta) \frac{1}{\sqrt{\Nc}}\Big] \lor %\Big[8\et( \Mc/\delta)\frac{1}{\sqrt{\Nc}}\Big] 
 \Big[32\et(2d/\delta)  \frac{d}{c\sqrt{\NF}}\Big] \lor (8/n).$$  
\begin{lemma}\label{lem:recFCP}
%    Write $\mathcal X, \hat x, \hat f$ for the parameters of $\FCP$ and $\mathcal X', \hat x', \hat f'$ for the actualised value used to recursively recall $\FCP$ after the iteration of the algorithm.    
    %Assume that Assumption~\ref{ass:noise} holds. 
    Assume that $\Cut$ is given a convex set $\mathcal X \subset \bar{\mathcal X}$, $\hat x \in \bar{\mathcal X}, \hat f \in \mathbb R, \delta>0$ such that:
    \begin{itemize}
        \item $|f(\hat x) - \hat f| \leq \et( \Mc/\delta)/\sqrt{\Nc}$
        \item {\bf either} $\mathcal X^* \subset \mathcal X$, {\bf or} $f (\hat x) - f^* \leq B.$
    \end{itemize}
    There exists an event of probability larger than $1-5\delta$ such that
    \begin{itemize}
    \item $\mathcal X'\subset \mathcal X$ is convex
    \item $|f(\hat x') - \hat f'| \leq \et(\Mc /\delta)/\sqrt{\Nc}$,
        \item {\bf either} $\Big[$ $\mathcal X^* \subset \mathcal X' \subset \mathcal X$ and $\mathrm{vol}(\mathcal X') \leq \frac{9}{10}\mathrm{vol}(\mathcal X')$ $\Big]$, {\bf or} $ f (\hat x') - f^* \leq B$,
        \item the total budget $T_{\Cut}$ used to run $\Cut$ until the next recursive call of $\Cut$ is such that
        $\Nc + \NF \leq T_{\Cut} \leq 4\Mc \Nc + \NF$.
    \end{itemize}
\end{lemma}
This lemma ensures that, provided that $\Cut$ is initialised properly, the convex set $\mathcal X'$ obtained after running $\Cut$ satisfies 
\begin{itemize}
    \item {\bf either} it contains $\mathcal X^*$, and its volume is a fraction of the volume of $\mathcal X$,
    \item {\bf or} $f$ measured at the current estimator of the minimum $\hat x'$ is already quite small.
\end{itemize}
The idea behind the proof of this lemma is that whenever $f(\hat x)- f^*$ is not too small, then by Proposition~\ref{prop:cut}, $\FCP$ will return with high probability a cutting point $z$ that satisfies the requirements in Lemma~\ref{lem:geom} - so that $\nabla g_c(z)$ is negatively correlated with $x^* - z$, and can therefore be used to cut the space $\mathcal X$. Also by Proposition~\ref{prop:cut}, with high probability $z$ is such that $\|z\|_2$ is small so that cutting the space according to this approximate center of gravity still preserves the nice property about exponentially fast volume reduction.

\paragraph{Step 4: Induction on several runs of $\Cut$.}  Based on this lemma, we proceed by induction over the repeated recursive runs of $\Cut$ after being called by $\algo$, conditioning over the high probability event of Lemma~\ref{lem:recFCP} where the condition for the next run are ensured. Our induction hypothesis $H_t$ is: on an event $\xi_t$ of probability larger than $1-5t\delta$, if $\Cut$ is called for the $t$ time, it takes as parameter a convex set $\mathcal X \subset \bar{\mathcal X}$, $\hat x \in \bar{\mathcal X}, \hat f \in \mathbb R$ such that:
    \begin{itemize}
        \item $|f(\hat x) - \hat f| \leq \et( \Mc/\delta)/\sqrt{\Nc}$
        \item {\bf either} $\mathcal X^* \subset \mathcal X$, {\bf or} $f (\hat x) - f^* \leq B.$
        \item the total budget $n_t$ used up to the $t$-th call of $\Cut$ is such that $(t-1)\Nc+ t\NF \leq n_t \leq 4(t-1)\Mc\Nc+ t\NF$.
    \end{itemize}
We prove this by induction:
\begin{itemize}
    \item Proof of $H_1$: Note first that by Lemma~\ref{lem:conc} and Lemma~\ref{lem:vol}, the conditions of Lemma~\ref{lem:recFCP} are satisfied after the initialisation phase of $\algo$ on an event of probability $1-\delta$. Moreover the running time of the initialisation is $\NF$. So $H_1$ holds.
    \item Proof of $H_{t+1}$ assuming that $H_t$ holds: assuming that $H_t$ holds for a given $t$, we have by Lemma~\ref{lem:recFCP} that $H_{t+1}$ holds on an event $\xi$ of probability larger than $1-\delta$, conditional on $\xi_t$. So writing $\xi_{t+1} = \xi_t \cap \xi$, we have proven that $H_{t+1}$ holds.
\end{itemize}
So for any given $t \geq 0$, on an event of probability larger than $1-5t\delta$
 if $\Cut$ is called for the $t$ time, it takes as parameter a convex set $\mathcal X \subset \bar{\mathcal X}$, $\hat x \in \bar{\mathcal X}, \hat f \in \mathbb R$ such that:
    \begin{itemize}
        \item $|f(\hat x) - \hat f| \leq \et(\Mc/\delta)/\sqrt{\Nc}$,
        \item {\bf either} $\Big[$ $\mathcal X^* \subset \mathcal X$ and $\mathrm{vol}(\mathcal X) \leq \Big(\frac{9}{10}\Big)^{t-1}\mathrm{vol}(\bar{\mathcal X})$ $\Big]$, {\bf or} $f (\hat x) - f^* \leq B.$
        \item the total budget $n_t$ used up to the $t$-th call of $\Cut$ is such that $t\NF \leq n_t \leq 2t\NF$ - since $4\Mc\Nc = \NF$. %\alex{To enforce} 
    \end{itemize}

\paragraph{Step 5: Application of the result of the induction to what happens at the end of the algorithm.} The induction from {\bf Step 4} applied to $t = \MF/5$ implies that, on an event of probability larger than $1 - 5(n/\NF)\delta = 1-\MF \delta$ - that we will write $\xi_{\mathrm{term}}$ - the algorithm $\algo$ terminates after at least $n/(2\NF)$ rounds, and at most $n/\NF$ rounds, and at its termination round, the current convex set $\mathcal X \subset \bar{\mathcal X}$, and the current value $\hat x$ (that $\algo$ will output as it is the last round) are such that 
\begin{itemize}
    \item {\bf either} $\Big[$ $\mathcal X^* \subset \mathcal X$ and $\mathrm{vol}(\mathcal X) \leq \Big(\frac{9}{10}\Big)^{n/(2\NF)-1}\mathrm{vol}(\bar{\mathcal X})$ $\Big]$,
    \item {\bf or} $f (\hat x) - f^* \leq B$.
\end{itemize}
If $f (\hat x) - f^* \leq B$, the proof is finished. So assume that on $\xi_{\mathrm{term}}$, we have $\mathcal X^* \subset \mathcal X$ and $\mathrm{vol}(\mathcal X) \leq \Big(\frac{9}{10}\Big)^{n/(2\NF)-1}\mathrm{vol}(\bar{\mathcal X})$. Note that as $\mathcal X$ is convex, we have $\mathrm{conv}(\mathcal X^*) \subset \mathcal X$ on $\xi_{\mathrm{term}}$.

By definition of $\NF$, we have that
$$\Big(\frac{9}{10}\Big)^{n/(2\NF)-1} \leq \Big[\frac{r}{nd\mathrm{diam}(\bar{\mathcal X})}\Big]^d.$$
So by Lemma~\ref{lem:vol}, have a contradiction on $\xi_{\mathrm{term}}$: $\mathrm{conv}(\mathcal X^*) \subset \mathcal X$, but $\mathrm{vol}(\mathrm{conv}(\mathcal X^*))\geq \mathrm{vol}( \mathcal X)$. So it means that on $\xi_{\mathrm{term}}$, we must have $f (\hat x) - f^* \leq B$.

%\alex{$$\NF \leq \frac{n\log(10/9) }{4d \log\Big(nd\mathrm{diam}(\bar{\mathcal X})/r\Big)} .$$}

\subsection{Proof of Proposition~\ref{prop:cut}}

In what follows write $\Mc:= \Mc(N)$. We first state the following lemma.
\begin{lemma}\label{lem:inducFCP}
    Assume that Assumption~\ref{ass:bound} holds. Consider $s\geq 0$ and $z\in \mathbb R^d$ such that $\mathbb B_2(z, 2c) \in \bar{\mathcal Z}^{\mathcal X}$.  There exists an event of probability larger than $1-3\delta$ such that the following holds on it, during a run of $\FCP(z, \mathcal X, c,N,s,\delta)$:
    \begin{itemize}
    \item  Assume that $g(z) - f^* \leq \frac{1}{N}$ and $s = 0$. For any $z_j^i$ such that 
    $$\hat g_{z_j^{(i)}}  - \hat g_z \geq  \frac{1}{16N}+4\et(2^i i^2/\delta) \sqrt{\frac{i^2 2^i}{N}},$$
    then since $2\et(2^i i^2/\delta) \sqrt{\frac{i^2 2^i}{N}} \geq \frac{(17/16)}{N}$, we have
    $$g(z_j^{(i)}) - f^*\geq  \frac{(17/16)}{N}.$$
        \item Assume that $g(z) - f^* \geq \frac{(17/16)^s}{N}$. For any $z_j^i$ such that 
    $$\hat g_{z_j^{(i)}}  - \hat g_z \geq  \frac{(17/16)^s}{16N}+4\et(2^i i^2/\delta) \sqrt{\frac{i^2 2^i}{N}},$$
    it holds that
    $$g(z_j^{(i)}) - f^* \geq  \frac{(17/16)^{s+1}}{N}.$$
    \item Assume that $g_c(2z) - g_c(z) \geq 2^{-3}(g_c(z) - f^*)$, and $g(z) - f^* > \frac{(17/16)^s}{N} \lor \Big[ 2^{15} \et(1/\delta)\log(2/\delta) \frac{1}{\sqrt{N}}\Big]$.%[ 4\et(1/\delta)/\sqrt{N}]$.
 Then there exists $z_j^i$ such that
 $$\hat g_{z_j^{(i)}}  - \hat g_z \geq  \frac{(17/16)^s}{16N}+4\et(2^i i^2/\delta) \sqrt{\frac{i^2 2^i}{N}}.$$
    \end{itemize}
\end{lemma}
Assume that $\mathbb B_2(z_0, 2\Mc c) \subset \bar{\mathcal Z}^{\mathcal X}$. Then by construction we know that even if $\FCP$ calls itself recursively for $\Mc$ rounds, then all the parameters $z$ that it will take at each round will be such that $\mathbb B_2(z, 2 c) \subset \bar{\mathcal Z}^{\mathcal X}$. Write $\tau$ for the random round where the recursive application of $\FCP(z_0, \mathcal X, c,N,0,\delta)$ stops. Applying Lemma~\ref{lem:inducFCP}, we know that on an event of probability larger than $1-3\Mc \delta$, for the point $z$ taken as input at round $\tau \land \Mc$:
\begin{itemize}
    \item if $\tau < \Mc$ we have either $$g_{2c}(z) - g_c(z) \leq 2^{-3}(g(z) - f^*),$$
or 
$$g(z) - f^* \leq  2^{15} \et(1/\delta)\log(2/\delta) \frac{1}{\sqrt{N}}.$$
This concludes the proof in this case.
\item otherwise if $\tau \geq \Mc$ then
$$g(z) - f^* \geq \frac{(17/16)^{\lfloor \Mc\rfloor}}{N}.$$
Note that by definition of $\Mc$ this implies $g(z) - f^* \geq 2$ which contradicts our Assumption~\ref{ass:bound}. So this case cannot happen. This concludes the proof in this case as well.
\end{itemize}

\subsection{Proof of Lemma~\ref{lem:recFCP}}

We first remind the following classical results of convex geometry.
\begin{proposition}[KLS Lemma]\label{lem:KLS}
    Let $\mathcal C$ be a convex in isotropic position. It holds that
    $$\mathbb B_2(1) \subset \mathcal C \subset \mathbb B_2(2d).$$
\end{proposition}

\begin{proposition}[Approximate barycentric cutting of an isotropic convex]\label{prop:approxiso}
    Let $\mathcal C$ be a convex in isotropic position. It holds for any $u \in \mathbb R^d: u\neq 0$, and any $z \in \mathbb R^d$
    $$\mathrm{vol}(\mathcal C \cap \{w: \langle w-z, u\rangle \geq 0\}) \geq (1/e - \|z\|_2)\mathrm{vol}(\mathcal C).$$
\end{proposition}
%Lemma 6.14 in http://sbubeck.com/Bubeck15.pdf

An immediate corollary of the last proposition is as follows
\begin{corollary}[Approximate barycentric cutting of a convex]\label{cor:approxiso}
    Let $\mathcal K$ be a convex. It holds for any $u \in \mathbb R^d: u\neq 0$, and any $z \in \mathbb R^d$
    $$\mathrm{vol}\Bigg(\mathcal K \cap F_{\mathcal X}^{-1}\Big(\{w: \langle w-z, u\rangle \geq 0\}\Big)\Bigg) \geq (1/e - \sqrt{d}\|z\|_2)\mathrm{vol}(\mathcal K).$$
\end{corollary}

From Proposition~\ref{lem:KLS} we deduce that
$$\mathbb B_2(2 \Mc c) \subset \mathcal Z^{\mathcal X}.$$
We therefore know that Proposition~\ref{prop:cut} holds for the output $(z, \hat g_z)$ of $\FCP(0, \mathcal X, c, \Nc,0, \delta)$ - and write $\xi$ for the event of probability larger than $1-4\delta$ where the proposition holds. Note already that it implies by definition of the algorithm that on $\xi$
$$|f(\hat x') - \hat f'| \leq \et(\Mc / \delta)/\sqrt{\Nc}, $$ 
and also that on $\xi$
$$\Nc + \NF \leq T_{\FCP} \leq 4\Mc \Nc + \NF,$$
and also that $\mathcal X'\subset \mathcal X$ is convex. Note also that on $\xi$ it implies by definition if $c$ that% \alex{$$c =  1/(8e\Mc\sqrt{d})$$}
$$\|z\|_2 \leq 2 \Mc c = 1/(4e\sqrt{d}),$$
which implies by Corollary~\ref{cor:approxiso}, by construction of the algorithm that on $\xi$
    $$\mathrm{vol}(\mathcal X \setminus \mathcal X') \geq \frac{1}{2e}\mathrm{vol}(\mathcal X),$$
    namely that on $\xi$
    $$\mathrm{vol}( \mathcal X') \geq (1 - \frac{1}{2e})\mathrm{vol}(\mathcal X).$$

\paragraph{Case 1: $g(z)$ is small, or $f(\hat x)$ is small.} We first consider the case where either $f(\hat x) - f^* \leq B$, or on $\xi$, we have that $g(z) - f^* \leq B$. In this case, we will have by definition of the algorithm that on $\xi$:
$$ f (\hat x') - f^* \leq B,$$
as $B \geq 2^{16} \et(\Mc/\delta)\log(2\Mc/\delta) \frac{1}{\sqrt{\Nc}} \geq  8\et(\Mc/\delta)\frac{1}{\sqrt{\Nc}}$. This concludes the proof.

\paragraph{Case 2: $g(z)$ and $f(\hat x) $ are large.} We now consider the converse case on $\xi$. In this case, we know by Proposition~\ref{prop:cut} that on $\xi$: 
$$g(z) - f^* \geq B.$$
In this case, we know by Proposition~\ref{prop:cut} that on $\xi$
$$g_{2c}(z) - g_c(z) \leq 2^{-3}(g_c(z) - f^*),$$
and we also know by assumption that
$$\mathcal X^* \subset \mathcal X.$$
So that by definition of $\mathcal X^*$, and since $\Nc \leq n$, for any $\tilde x \in \mathcal X^*$, on $\xi$
$$g_{2c}(z) - g_c(z) \leq 2^{-2}(g_c(z) - g(F_{\mathcal X}(\tilde x))).$$
We can therefore apply Lemma~\ref{lem:geom}, and we have that on $\xi$ 
    \begin{equation}\label{eq:devtomin}
        \langle \nabla g_c(z), z - F_{\mathcal X}(\tilde x) \rangle \geq \frac{3}{4}[g_c(z)) - g(F_{\mathcal X}(\tilde x))] \geq \frac{5}{8}[g_c(z)) -f^*] \geq \frac{5}{8} B,
    \end{equation}
    as $B \geq 8/n$ and $g(F_{\mathcal X}(\tilde x)) - f^* \leq 1/n$. %\alex{$B \geq 8/n$}

Also, by Lemma~\ref{lem:conc}, for any $u \in \mathbb R^d$, conditional to $\xi$ and on an event $\xi'$ of probability larger than $1-\delta$
$$|\langle \widehat{ \nabla g_c} - \nabla g_c(z), u\rangle|  \leq \et(1/\delta) \|u\|_2  \frac{\sqrt{d}}{c\sqrt{\NF}}.$$
So that on $\xi'\cap \xi$, for any $\tilde x \in \mathcal X^*$
$$|\langle \widehat{ \nabla g_c} - \nabla g_c(z), z - F_{\mathcal X}(\tilde x)\rangle|  \leq \et(2d/\delta) \|z - F_{\mathcal X}(\tilde x)\|_2  \frac{\sqrt{d}}{c\sqrt{\NF}}.$$
From Lemma~\ref{lem:KLS} this implies on $\xi'\cap \xi$
$$|\langle \widehat{ \nabla g_c} - \nabla g_c(z), z - F_{\mathcal X}(\tilde x)\rangle|  \leq 2\et(2d/\delta)  \frac{d}{c\sqrt{\NF}} \leq B/16,$$
as $B \geq  32\et(2d/\delta)  \frac{d}{c\sqrt{\NF}} $.

Combining this result with Equation~\eqref{eq:devtomin} leads to the fact that on $\xi'\cap \xi$, for any $\tilde x \in \mathcal X^*$
 \begin{equation*}
        \langle \widehat{ \nabla g_c}, z - F_{\mathcal X}(\tilde x) \rangle \geq \frac{1}{16} B.
    \end{equation*}
    So that on $\xi'\cap \xi$, we have that $\mathcal X^* \subset \mathcal X'$. This concludes the proof.

%that with probability larger than $1-\delta$, the number of iterations needed by $\Cut()$

\begin{proof}[Proof of Lemma~\ref{lem:inducFCP}]

By Lemma~\ref{lem:conc}, it holds on an event of probability larger than $1 - \delta (1 + \sum_k 1/k^2) \geq 1 - 2.5\delta$ that
$$|g( z)  - \hat g_z| \leq \et(1/\delta)/\sqrt{N},$$
and for any $i \leq I_N$, $j \leq 2^i$
$$|g(z_j^{(i)})  - \hat g_{z_j^{(i)}}| \leq \et(2^i i^2/\delta)/\sqrt{N}.$$
Write $\xi$ for this event. 

Note that if $g(z) - f^* \leq \frac{1}{N}$ and $s=0$, then on $\xi$ we have that if there exists $z_j^i$ such that 
    $$\hat g_{z_j^{(i)}}  - \hat g_z \geq  \frac{1}{16N}+4\et(2^i i^2/\delta) \sqrt{\frac{i^2 2^i}{N}},$$
    then since $2\et(2^i i^2/\delta) \sqrt{\frac{i^2 2^i}{N}} \geq \frac{(17/16)}{N}$, we have
    $$g(z_j^{(i)}) - f^*\geq  \frac{(17/16)}{N}.$$
    The first part of the lemma is therefore proven.

Assume now that $g(z) - f^* \geq \frac{(17/16)^s}{N}$. Note first that on $\xi$, we have that if there exists $z_j^i$ such that 
    $$\hat g_{z_j^{(i)}}  - \hat g_z \geq  \frac{(17/16)^s}{16N}+4\et(2^i i^2/\delta) \sqrt{\frac{i^2 2^i}{N}},$$
  then since $g(z) - f^* \geq \frac{(17/16)^s}{N}$
    $$g(z_j^{(i)}) - f^*\geq  \frac{(17/16)^{s+1}}{N}.$$
    The second part of the lemma is therefore proven. 
    
    Now assume that $z$ satisfies the conditions of the third part of the lemma, namely $g_c(2z) - g_c(z) \geq 2^{-3}(g_c(z) - f^*)$, and $g(z) - f^* > \frac{(17/16)^s}{N} \lor \Big[ 2^{15} \et(1/\delta)\log(2/\delta) \frac{1}{\sqrt{N}}\Big]$.

\paragraph{Step 1: establishing condition under which at least a $z_j^{(i)}$ is selected.} 

On $\xi$ it holds that
$$|g(z_j^{(i)}) - g(z)| \leq 2\et(2^i i^2/\delta) \sqrt{\frac{i^2 2^i}{N}}.$$
So if
\begin{equation}\label{eq:slec}
    g(z_j^{(i)}) - g(z) \geq \frac{(17/16)^s}{16N} + 6\et(2^i i^2/\delta) \sqrt{\frac{i^2 2^i}{N}}:= \Delta_i,
\end{equation}
then on $\xi$ it can be selected as it satisfies
$$\hat g_{z_j^{(i)}} - \hat g_z \geq \frac{(17/16)^s}{16N} + 4\et(2^i i^2/\delta) \sqrt{\frac{i^2 2^i}{N}}.$$

We now recall the following concentration result.
\begin{lemma}[Concentration of Binomial random variables] \label{lem:bin}\alex{tocite}
    Let $p\in [0,1]$ and $m\geq 1$. Let $X_1, \ldots, X_m \sim_{i.i.d.} \mathcal B(p)$. Then with probability larger than $1-\delta$
    $$|\frac{1}{m} \sum_{i=1}^m X_i - p| \leq \sqrt{2p\frac{\log(2/\delta)}{m}} + 2\frac{\log(2/\delta)}{m},$$
    which implies in particular that with probability larger than $1-\delta$
    $$\frac{p}{2} - 4\frac{\log(2/\delta)}{m} \leq \frac{1}{m} \sum_{i=1}^m X_i  \leq 2p + 2\frac{\log(2/\delta)}{m}.$$
\end{lemma}

Assume that there exists $i \leq I_N$ such that
$$\mathbb P_{Z\sim \mathcal U_{\mathbb B_2(2c)}}(g(z+Z) - g(z) \geq \Delta_i) > 8\frac{\log(2/\delta)}{2^i}.$$
By Lemma~\ref{lem:bin}, then we know that with probability larger than $1-\delta$, at least one of the $z_j^{(i)}$ for some $j$ will be such that 
$$g(z_j^{(i)}) - g(z) \geq \Delta_i.$$
Using Equation~\eqref{eq:slec}, we therefore know that in this case, with probability larger than $1-4\delta$, $z_j^{(i)}$ will be selected, finishing the proof in this case.

\paragraph{Step 2: Converse case where for any $i \leq I_N$, we have $\mathbb P_{Z\sim \mathcal U_{\mathbb B_2(2c)}}(g(z+Z) - g(z) \geq \Delta_i) \leq  8\frac{\log(2/\delta)}{2^i}$.}

We remind that
$$\Delta_i = \frac{(17/16)^s}{8N} + 6\et(2^i i^2/\delta) \sqrt{\frac{i^2 2^i}{N}}.$$
Note that by assumption, we therefore have that for any $i \leq I_N$, we have 
$$\mathbb P_{Z\sim \mathcal U_{\mathbb B_2(2c)}}\Bigg(g(z+Z) -  g(z) - \frac{(17/16)^s}{16N}  \geq 6\et(2^i i^2/\delta) \sqrt{\frac{i^2 2^i}{N}}\Bigg) \leq  8\frac{\log(2/\delta)}{2^i}.$$
So that, since $4\et(2^{I_N} I_N^2/\delta) \sqrt{\frac{I_N^2 2^{I_N}}{N}} \geq 2$ by definition of $I_N$
$$\mathbb E_{Z\sim \mathcal U_{\mathbb B_2(2c)}}\Bigg[g(z+Z) -  g(z) - \frac{(17/16)^s}{16N} \Bigg] \leq \sum_{i \leq I_N} 64\et(2^{i+1} (i+1)^2/\delta) \sqrt{\frac{(i+1)^2 2^{i+1}}{N}} \times  \frac{\log(2/\delta)}{2^i},$$
leading to
\begin{align*}
g_{2c}(z) -  g(z) - \frac{(17/16)^s}{16N} &\leq 64\sqrt{2} \et(1/\delta)\log(2/\delta) \frac{1}{\sqrt{N}} \sum_{i \leq I_N} (i+1)^4 2^{-i/2}   \\
&\leq 2^{11} \et(1/\delta)\log(2/\delta) \frac{1}{\sqrt{N}}.
\end{align*}
Since $g(z) - f^* \geq \frac{(17/16)^s}{N}$, and $g(z) - f^* > 2^{15} \et(1/\delta)\log(2/\delta) \frac{1}{\sqrt{N}}$ by assumption, then we have that
\begin{align*}
g_{2c}(z) - f^*
&< (g(z) - f^*)(1+2^{-3}).
\end{align*}
This implies
$$g_{2c}(z) - g_c(z) < 2^{-3}(g(z) - f^*) \leq 2^{-3}(g_c(z) - f^*).$$
This contradicts our assumption that $g_{2c}(z) - g_c(z) \geq 2^{-3}(g_c(z) - f^*)$, so that this case cannot happen under our assumption. This concludes the proof.
%On $\xi_1$ we have that
%$$\frac{(17/16)^s}{8N}+2\et(2^i i^2/\delta) \sqrt{\frac{i^2 2^i}{N}}$$

%\begin{lemma}
%    Consider a random variable $X$ taking values in $[-1,1]$, and $x_0 \in \mathbb R$ and such that, for a collection of numbers $0 \leq m_1 < m_2 < \ldots < m_K$

    %For any $1\geq \Delta >0$, there exists $i^*\leq \log_2(2/\Delta)$ such that
%$$\mathbb P(X \geq 2^{-i^*}\Delta) \geq 2^{i^*-2}i^{-2}\Delta^{-1} \mathbb E X.$$
%    \end{lemma}
%\begin{proof}
%    For any $1\geq \Delta >0$
%$$\mathbb E X \leq  \sum_{i\leq \log_2(2/\Delta)} 2^{-i+1}\Delta\mathbb P(X \geq 2^{-i}\Delta).$$
%So since $\sum_i 1/i^2 \leq 2$, there exists $i^*\leq \log_2(2/\Delta)$ such that
%$$2^{-i^*+2}\Delta\mathbb P(X \geq 2^{-i^*}\Delta) \geq \mathbb E X/i^{*2}.$$
%So that
%$$\mathbb P(X \geq 2^{-i^*}\Delta) \geq 2^{i^*-2}i^{-2}\Delta^{-1} \mathbb E X.$$
%\end{proof}

\end{proof}

\subsection{Proof of technical lemmas}\label{lem:tech}

\begin{proof}[Proof of Lemma~\ref{lem:geom}]
Assume without loss of generality that $\tilde z = 0$ and $g(\tilde z) = 0$. Let $c>0$ and $z \in \mathbb R^d$ 
%\in \bar{\mathcal Z}$ such that $\mathbb B_2(z,2c) \subset \bar{\mathcal Z}$ and 
such that $g_{2c} (z) - g_c(z) \leq 2^{-2} g_c(z)$. In order to prove the lemma, it suffices to prove that $\langle \nabla g_c(z), z \rangle \geq 3g_c(z))/4$.

    By convexity of $g$ on $\mathbb R^d$, note that for any $z'\in \mathbb R^d$, we have $g(z'/2) \leq g(z')/2$. So that by definition of $g_c$:%, if $\mathbb B_2(z',2c) \subset \bar{\mathcal Z}$:
    $$g_{c}(z'/2) \leq g_{2c}(z')/2.$$
    
    Since $g_{2c} (z)  \leq (5/4) g_c(z)$, we have applying the formula above to $z$
    $$g_{c}(z/2) \leq g_{2c}(z)/2 \leq 5g_{c}(z)/8,$$
    so that
    \begin{equation}\label{eq:geom1}
        g_c(z) - g_{c}(z/2) \geq 3g_{c}(z)/8.
    \end{equation}

    Since $g_c$ is convex and differentiable on $\mathbb R^d$% any point $z' \in \mathbb R^d$, % such that$\mathbb B_2(z',c) \subset \bar{\mathcal Z}$, 
%    we know that for any two points $z',z''\in \mathbb R^d$:% such that $\mathbb B_2(z',c) \subset \bar{\mathcal Z}$, $\mathbb B_2(z'',c) \subset \bar{\mathcal Z}$:
%    $$g_c(z') - g_c(z'') \leq \langle \nabla g_c(z'), z' - y'' \rangle.$$
%    So that applying this to $z' = z, z'' = z/2$:
    $$g_c(z) - g_c(z/2) \leq \langle \nabla g_c(z), z/2 \rangle.$$
    So that finally by Equation~\eqref{eq:geom1}:
        $$3g_{c}(z)/4 \leq \langle \nabla g_c(z), z \rangle.$$
\end{proof}

\begin{proof}[Proof of Lemma~\ref{lem:conc}]

\noindent
\textbf{Bound on the deviations of $\hat g_c(z)$.} Let $\delta >0$. Note that $\hat g_c(z)$ is the empirical mean of the $\tilde y_i^{(b)}$, which are by construction i.i.d.~random variables such that $\tilde y_i^{(b)}\in [0,1]$ and $\mathbb E[\tilde y_i^{(b)}] = g_c(z)$. So that, applying Hoeffding's inequality (see e.g.~\cite[Theorem 2.8]{boucheron2003concentration}), with probability larger than $1-\delta$
$$|\hat g_c(z) - g_c(z)|  \leq \sqrt{\frac{\log(2/\delta)}{2N}},$$
leading to the result.

\noindent
\textbf{Bound on the deviations of $\langle \widehat{ \nabla g_c}(z)$.} Let $\delta >0$. Note now that $\mathbb E\langle \widehat{ \nabla g_c}(z) , u\rangle$ is the empirical of the i.i.d.~random variables $W_i := \frac{d}{c^2} \tilde y_i^{(s)} \langle Z_i^{(s)},  u \rangle$. Note that by Equation~\eqref{eq:nableq}, we have
$$\mathbb E W_i = \langle \nabla g_c(z), u\rangle,$$
and
\begin{align*}
    \mathbb E W_i^2 &= \frac{d^2}{c^4} \mathbb E [(\tilde y_i^{(s)})^2 \langle Z_i^{(s)},  u \rangle^2]\\
    &\leq \frac{d^2}{c^4} \mathbb E [\langle Z_i^{(s)},  u \rangle^2]\\
    &= \frac{d^2}{c^4} \frac{c^2}{d} \|u\|_2^2 = \frac{d}{c^2} \|u\|_2^2,
\end{align*}
and
\begin{align*}
    |W_i| &= \frac{d}{c^2} \tilde y_i^{(s)} |\langle Z_i^{(s)},  u \rangle|\\
    &\leq \frac{d}{c} \|u\|_2.
\end{align*}
So that, applying Bernstein's inequality (see e.g.~\cite[Theorem 2.10]{boucheron2003concentration}), with probability larger than $1-\delta$
$$||\langle \widehat{ \nabla g_c}(z) - \nabla g_c(z), u\rangle||  \leq \frac{\sqrt{d}}{c} \|u\|_2 \sqrt{2\frac{\log(2/\delta)}{N}} +  2\frac{d}{c} \|u\|_2\frac{\log(2/\delta)}{N}.$$
Since $N \geq d$, this leads to the result.
\end{proof}

\begin{proof}[Proof of Lemma~\ref{lem:vol}]
We know that $\mathbb B_1(x^*, r) \in \bar{\mathcal X}$. So that for any $i \leq d$, we have
$$f(x^*+re_i) \leq 1 , f(x^*-re_i) \leq 1,$$
which implies by convexity of $f$ and $\bar{\mathcal X}$ that
$$f(x^*+re_i/n) \leq 1/n, f(x^*-re_i/n) \leq 1/n,$$
and these points are in $\bar{\mathcal X}$.
Note that
$$\mathrm{vol}(\mathbb B_1(x^*, r/n)) \geq (r/(nd))^d.$$
Note also that
$$\mathrm{vol}(\bar{\mathcal X}) \leq \mathrm{diam}(\bar{\mathcal X})^d.$$
So that since $\mathbb B_1(x^*, r/n) = \mathrm{conv}(\mathcal X^*)$
$$\frac{\mathrm{vol}(\mathrm{conv}(\mathcal X^*)}{\mathrm{vol}(\bar{\mathcal X})} \geq \Big[\frac{r}{nd\mathrm{diam}(\bar{\mathcal X}}\Big]^d.$$
\end{proof}

\begin{proof}[Proof of Corollary~\ref{cor:approxiso}]
    The corollary follows from Proposition~\ref{prop:approxiso} using the facts that $F_{\mathcal X}(\mathcal K)$ is in istropic position rescaled by $d^{-1/2}$ and also that since $F_{\mathcal X}^{-1}$ is a linear application of the form $F_{\mathcal X}^{-1}(z) = \sqrt{d} \Sigma_{\mathcal X}^{1/2} (z+ \mu_{\mathcal X})$, then for any convex $\mathcal K'$
    $$\mathrm{vol}(F_{\mathcal X}^{-1}(\mathcal K')) = d^{d/2} \mathrm{det}(\Sigma_{\mathcal X})^{1/2}\mathrm{vol}(\mathcal K'),$$
    where $\mathrm{det}(\Sigma_{\mathcal X})$ is the determinant of $\Sigma_{\mathcal X}$.
\end{proof}

\end{document}